\documentclass[12pt, a4paper,reqno]{amsart}
\usepackage{amscd,amssymb,amsmath, array,latexsym,amsbsy}
\usepackage[active]{srcltx}
\def\Cl{\operatorname{Cl}}

\def\R{\mathbb{R}}

\def\NN{\mathbb{N}}

\newtheorem{thm}{Theorem}[section]
\newtheorem{cor}[thm]{Corollary}
\newtheorem{prop}[thm]{Proposition}
\newtheorem{lemma}[thm]{Lemma}
\theoremstyle{definition}

\numberwithin{equation}{section}
\title[Pure Quotients]{Pure quotients and Morita's theorem for $k_{\omega}$-spaces}

\author
[Lazar]{Aldo J. Lazar}
\address{School of Mathematical Sciences
\\Tel Aviv University
\\Tel Aviv 69978
\\Israel
} \email{aldo@tauex.tau.ac.il}

\author[Somerset]{Douglas W.B. Somerset}
 \email{douglassomerset@yahoo.com}

\begin{document}

\begin{abstract}
A $k_{\omega}$-space $X$ is a Hausdorff quotient
of a locally compact, $\sigma$-compact Hausdorff space. A theorem
of Morita's describes the structure of $X$ when the quotient map is closed, but in 2010 a question of Arkhangel'skii's highlighted
the lack of a corresponding theorem for non-closed quotient maps (even from subsets of $\R^n$).
Arkhangel'skii's specific question had in fact been answered by Siwiec in 1976, but a general structure theorem
for $k_{\omega}$-spaces is still lacking. We introduce pure quotient maps, extend Morita's
theorem to these, and use Fell's topology to show that every quotient map can be `purified' (and  thus
every $k_{\omega}$-space is the image of a pure quotient map).
This clarifies the structure of arbitrary $k_{\omega}$-spaces and gives a fuller answer to
Arkhangel'skii's question.
\end{abstract}

\maketitle

\thanks{}



\bigskip

\section{Introduction}

\noindent A $k_{\omega}$-space -- or a hemicompact $k$-space -- is a Hausdorff space which is the image
of a locally compact, $\sigma$-compact Hausdorff space under a quotient map.
The class of $k_{\omega}$-spaces was extensively studied from the 1940s and the literature
of the subject (see \cite{Fr} for a useful summary)
contains some famous names: R. Arens, M. Graev, E. Michael,
J. Milnor, K. Morita, N. Steenrod, and others. Since then
$k_{\omega}$-spaces have become a standard tool (like locally compact Hausdorff spaces).

For general topologists, the class of $k_{\omega}$-spaces has good closure properties (under
quotients, closed subspaces, and finite products). In the study of spaces of continuous functions, the $k_{\omega}$-spaces
are those for which the space $C_k(X)$ of continuous real-valued functions on $X$, with the compact-open topology,
is completely metrizable (see \cite{MN} for example).  In C$^*$-algebras (the authors' interest)  the Glimm spaces of $\sigma$-unital
C$^*$-algebras belong to this class \cite[Theorem 2.6]{Laz}.
In topological algebras, the spectrum of a Fr\'echet algebra is usually (but not always) a $k_{\omega}$-space; and $k_{\omega}$-spaces crop up, too, in the study of topological
groups and semi-groups, and in mathematical economics, see \cite{CCM} for example.

A $k_{\omega}$-space is paracompact and normal \cite[Lemma 5]{Mori}, but can be nowhere first countable,
even when it is the quotient of a second countable space. The standard example is the countable space
$S_{\omega}$ introduced by Arkhangel'skii and Franklin \cite{AF}. One well-known weakening of first countability
is the {\sl Fr\'echet-Urysohn} property: a topological space $X$ is {\sl Fr\'echet-Urysohn} at $x\in X$ if $A\subseteq X$
and $x\in\overline A$ implies the existence of a sequence $(x_n)_{n\ge 1}\subseteq A$ with $\lim_n x_x=x$. If
$X$ is Fr\'echet-Urysohn at each point then $X$ is a {\sl Fr\'echet-Urysohn space}. The
sequential fan (consisting of  countably many convergent sequences with the non-isolated points identified to a single point)
is a countable Fr\'echet-Urysohn $k_{\omega}$-space with first countability failing at a single point.
Attempts to build more complicated examples run into difficulties, however, and in 2010, in exhibiting `a countable Tychonoff Fr\'echet-Urysohn
space which is nowhere first countable', Arkhangel'skii raised the question of whether a quotient of a locally compact second countable metric space
could be found with all these features \cite{Arh}. In point of fact this question had already been answered negatively by Siwiec in 1976 \cite{Siw},
but the fact that the question was even asked is significant because it highlights a basic gap in the understanding
of $k_{\omega}$-spaces.

The general picture is as follows. Recall that a point $x$ in a topological space $X$ is a {\sl $k$-point}
if $E\subseteq X$ and $x\in\overline E$ implies that there is a compact set $K$ such that $x\in \overline{E\cap K}$. Clearly
every point with a compact neighbourhood is a $k$-point.
As temporary notation, for a $k_{\omega}$-space $X$, let $L$ denote the set of points with compact neighbourhood, $F$ the set
of $k$-points without compact neighbourhood, and $N$ the set of non-$k$-points.
Then $L$ is obviously an open set, and Morita and Arkhangel'skii (Theorems 2.1 and 2.3 below)
showed that if $X$ is the image of a closed quotient map then $L$ is dense, $F$ is discrete, and $N$ is empty. Furthermore, the points of $L$
are characterised by the fact that their inverse images have compact boundary.

When a $k_{\omega}$-space $X$ is the image of a general (non-closed) quotient map $q$, the set $L$ can still be characterised by the biquotient property
and $L\cup F$ by pseudo-openness (definitions below), but there is no explicit extension of
Morita's theorem relating these properties to $q$ being locally closed or attempting to describe the sets $F$ and $N$.

In this paper we introduce pure quotients maps -- a partial generalisation of closed quotient maps -- and
show that whenever $X$ is the image of a quotient map $q$ from a locally compact $\sigma$-compact Hausdorff space $Y$,
there is a corresponding pure quotient map $q_*:Y_*\to X$ where $Y_*$ is  a locally compact $\sigma$-compact Hausdorff space derived from $Y$ using
Fell's topology on closed sets. Thus every $k_{\omega}$-space is the image of a pure quotient map.

For pure quotient maps, $N$ is precisely the set of points at which the map is not locally closed,
and we extend Morita's theorem by showing that $F$ is contained in the set of P-points of
    $N\cup F$ and is thus a P-space with empty interior in $X$ (Theorem 4.5). Thus a general $k_{\omega}$-space $X$ decomposes as the
disjoint union $L\cup F\cup N$ where $L$ is open and $F$ has empty interior and is contained in the set of P-points of
the closed set $F\cup N$. If $X$ is countably tight (in particular if $X$ is a quotient
of a locally compact subset of ${\R}^n$) then the P-points of $F\cup N$ are isolated points of $F\cup N$, lying in the closure of $L$; so $N$ is closed, $L\cup F$ is open
with $L$ dense, and $F$ is discrete (Corollary 4.6).
For general $X$, however, $F$ need not be discrete, nor $N$ closed, nor $F$ lie in the closure of $L$; and
we show by example that $F\cup N$ can be any compact Hausdorff space with $F$ as its set of P-points (Example 4.7).
A Fr\'echet-Urysohn space is countably tight with $N$ empty, so the answer to Arkhangel'skii's question quickly follows
from the description just given (see Corollary 5.3 and Theorem 5.5, where more general results are obtained).

The authors encountered pure quotient maps in work on Glimm spaces of C$^*$-algebras, where they occur naturally
(indeed Theorem 4.5, in the second countable case, was originally proved entirely using C$^*$-algebras).
They look less obvious when translated into a topological context, but perhaps that is part of their wider interest.


\section{Morita's theorem and locally closed maps}

\noindent Let $Y$ be a locally compact, $\sigma$-compact Hausdorff space
and $q:Y\to X$ a quotient map with $X$ Hausdorff. Then we can write $Y=\bigcup_{i\ge 1} Y_i$ where each $Y_i$ is compact and is contained in the
interior of its successor. We will say that $\{Y_i: i\ge 1\}$ is a {\sl compact decomposition} for $Y$, while the compact sets $\{ q(Y_i): i\ge 1\}$
are a {\sl $k_{\omega}$-decomposition} for $X$. The space $X$ is {\sl hemicompact} with regard to a $k_{\omega}$-decompositon; that is,
for any compact $K\subseteq X$, eventually $K\subseteq q(Y_i)$ for some $i$ (see \cite[p. 113]{Fr}).

\smallskip
One of the central results on $k_{\omega}$-spaces is due to Morita \cite[Theorem 4]{Mor}.

\begin{thm}\label{Theorem 2.1} {\rm(Morita, 1956)}
Let $Y$ be a locally compact, $\sigma$-compact Hausdorff space and $q:Y\to X$
a quotient map with $X$ Hausdorff. If $q$ is closed then $X$ is locally compact except
at a closed discrete set of exceptional points, namely those points $x$ for which the fibre $q^{-1}(x)$ has non-compact
boundary.
\end{thm}
\smallskip
\noindent As Morita observes, the discrete set of exceptional points is in fact contained in a larger discrete set of points with non-compact
fibre. For all other points, the fibre is compact.

\medskip
\noindent If $q$ is not closed, the situation becomes more complicated, and to discuss this the following definition is useful.
Let $X$ and $Y$ be topological spaces. A map $q:Y\to X$ is {\sl locally closed} at a point $x\in X$ if for every closed set
$W\subseteq Y$, $x\in\overline {q(W)}$ implies $x\in q(W)$ \cite[\S 13.XIV]{Kur}. Clearly if $q$ surjective then $q$ is closed if and only
$q$ is locally closed at each point of $x$. One basic result is as follows.

\begin{prop}\label{Proposition 2.2.} Let $Y$ be a locally compact Hausdorff space and $q:Y\to X$ a continuous surjective map
with $X$ Hausdorff. Let $x\in X$ and suppose that $q$ is locally closed at $x$ and that the boundary of $q^{-1}(x)$
is compact. Then $x$ has a compact neighbourhood in $X$.
\end{prop}

\begin{proof} If the boundary of $q^{-1}(x)$ is compact then it is contained in an open subset $O$
of $Y$ whose closure is compact. The set $U:=q^{-1}(x)\cup O$ is open and the map $q$
is locally closed at $x$. Thus $V:=X\setminus \overline{q(Y\setminus U)}$ is an open neighbourhood of $x$.
We have $q^{-1}(x)\subseteq q^{-1}(V)$, hence $\overline V\subseteq\overline{q(U)}=\{x\}\cup\overline{q(O)}=
\{x\}\cup q(\overline O)$, and thus $\overline V$ is compact.
\end{proof}

\medskip
\noindent In general, there seems to be little relation between the compactness of the fibre and
`local closedness' of the quotient map on the one hand, and the local compactness of the image on the other.
The following example is typical. Let $Y=\bigcup_{n\ge 0} L_n$ where $L_n$ is the line $y=n$ in $ {\R}^2$, and let
$q:Y\to X$ be the projection of $Y$ onto the $x$-axis, which is a quotient map. Then $X$ is locally compact and Hausdorff
but $q$ is nowhere locally closed and $q^{-1}(x)$ has non-compact boundary for every $x\in X$.

\medskip

A rather similar definition to `locally closed' was introduced independently by various authors, the earliest of whom
seems to have been McDougle \cite{McD}. Let $X$ and $Y$ be topological spaces.
A surjective map $f:Y\to X$ is {\sl pseudo-open at $x\in X$} if whenever $U$ is an open subset
of $Y$ containing $f^{-1}(x)$, $f(U)$ is a neighbourhood of $x$.  If $f$ is pseudo-open
at every point of $X$ then $f$ is said to be {\sl pseudo-open}.  A pseudo-open map is easily seen to be a quotient map,
and if $f$ is locally closed at $x$ then $f$ is pseudo-open. The example above, with $q$ open but nowhere locally closed,
shows that the converse is not true. We will see, however, that the two conditions are equivalent if the quotient map $q$ is pure.

\smallskip

The next theorem was essentially proved in \cite[Theorems 3.3 and 3.4]{Arhang} (see also \cite{Arha}) but our statement of it is somewhat different, so we
include the short proof here.
\medskip

\begin{thm}\label{Theorem 2.3} {\rm (Arkhangel'skii, 1963)}  Let $Y$ be a locally compact $\sigma$-compact Hausdorff space and $q:Y\to X$ a quotient
map with $X$ Hausdorff. Let $x\in X$. Then $x$ is a $k$-point if and only if $q$ is pseudo-open at $x$.
\end{thm}

\begin{proof} Suppose first that $q$ is pseudo-open at $x$ (this does not use the $\sigma$-compactness
of $Y$). Let $E\subseteq X$ with $x\in\overline E$, and set
$F=q^{-1}(E)$. Then $\overline F$ meets $q^{-1}(x)$, at $y$ say, because $q$ is pseudo-open (for otherwise there is
an open set $U$ containing $q^{-1}(x)$ and disjoint from $F$, and then $q(U)$ is a neighbourhood of $x$
disjoint from $E$). Let $L$ be a compact neighbourhood of $y$ and set $K=q(L)$. Then $y\in \overline{F\cap L}$
and $x\in\overline{E\cap K}$.

Conversely, suppose that $x$ is a $k$-point and let $U$ be an open subset
of $Y$ containing $f^{-1}(x)$. Let $E=X\setminus q(U)$ and $F=q^{-1}(E)$. Then $F$ does not meet $U$. If $x\in\overline E$
then there exists a compact set $K$ such that $x\in \overline{E\cap K}$. By hemicompactness, there exists
compact $L$ in $Y$ such $q(L)=K$. Then $\overline{F\cap L}$ is compact and $x\in q(\overline{F\cap L})$. Hence
$\overline F$ meets $q^{-1}(x)$, contradicting the fact that $F$ does not meet $U$. Thus $x\notin \overline E$, and
$q(U)$ is a neighbourhood of $x$.
\end{proof}

\medskip
\noindent For a quotient map $q: Y\to X$, let $H_q$ be the set of points in $X$ at which $q$ is pseudo-open.
Theorem 1.3 shows that if $Y$ is a locally compact $\sigma$-compact Hausdorff space and $X$ is Hausdorff
then $H_q$ is the set of $k$-points in $X$.

Among the $k$-points, the
points $x$ of local compactness in a $k_{\omega}$-space (and more generally) can be characterised as those which satisfy the {\sl biquotient} property: every open
cover $\mathcal {U}$ of $q^{-1}(x)$ has a finite subset $\{U_1, \ldots, U_n\}\subseteq \mathcal{U}$ such that $\{ q(U_1), \ldots, q(U_n)\}$ covers
a neighbourhood of $x$ (see \cite{Stone}, \cite{Mich}). While the characterisations given by the biquotient property and by pseudo-openness are elegant and versatile, they
do not seem to lead to a structure theory for general $k_{\omega}$-spaces along the lines of Morita's theorem. Something is needed to bring the ideas together.

\section{Pure quotient maps}

\noindent  In this section we introduce pure quotient maps. They are weaker than closed
quotient maps in one important aspect, but with a compensatory strengthening in another direction; and while they retain some of the nice properties
of closed quotient maps, they are much more general. We begin with a preparatory lemma.

\begin{lemma}\label {Lemma 3.1.} Let $Y$ be a locally compact $\sigma$-compact Hausdorff space and $q:Y\to X$ a quotient map with $X$ Hausdorff space.
Then the following are equivalent:

(i) $q$ is not closed;

(ii) there are nets $(y_{\alpha})$ (with the points $y_{\alpha}$ from distinct fibres) and $(z_{\alpha})$ in $Y$ with $y_{\alpha}\to \infty$ and
$z_{\alpha}\not\to\infty$ such that $q(y_{\alpha})=q(z_{\alpha})$ for all $\alpha$;

(iii) there are sequences $(y_n)$ (with the points $y_n$ from distinct fibres) and $(z_n)$ in $Y$ with $y_n\to \infty$ and $(z_n)$
contained in a compact set such that $q(y_n)=q(z_n)$ for all $n$.
\end{lemma}

\begin{proof} (i)$\Rightarrow$(ii). Let $W$ be a closed subset of $Y$ such that $q(W)$ is not closed.
Then the saturation $W_0$ of $W$ is not closed, so there is a net $(z_{\alpha})$ in $W_0$, which can be chosen from distinct fibres, such that $z_{\alpha}\to w\notin W_0$.
Let $(y_{\alpha})$ be a corresponding net in $W$ with $y_{\alpha}$ and $z_{\alpha}$ from the same fibre for each $\alpha$. Then if $(y_{\alpha})$
had any convergent subnet $(y_{\beta})$ with limit $y\in W$, $y$ and $w$ would lie in the same fibre so $w$ would belong to $W_0$, a contradiction. Hence
$y_{\alpha}\to \infty$.

(ii)$\Rightarrow$(iii). If (ii) holds, then $(z_{\alpha})$ is frequently in some compact set $K$, so for each $n\ge 1$, we may choose $z_n$
from the set $\{ z_{\alpha}\}$ such the corresponding $y_n\in \{ y_{\alpha}\}$ does not belong to $K_n$ (where $Y=\bigcup_{n\ge 1} K_n$ is a compact
decomposition for $Y$). Then $y_n\to\infty$ while $(z_n)$ is contained in $K$.

(iii)$\Rightarrow$(i). If (iii) holds, let $K$ be the compact set containing $(z_n)$. Then $(z_n)$ has subnet
converging to some $y\in K$, so $q(y)$ lies in the closure of the set $\{q(z_n): n\ge 1\}$. But $\{q(z_n):n\ge 1\}=q(\{y_n: n\ge 1\})$, the image of
a closed set, so $q$ is not a closed map.
\end{proof}

\bigskip
\noindent {\bf Definition:}
Let $Y$ be a locally compact $\sigma$-compact space and $q:Y\to X$ a quotient map with $X$ Hausdorff.
We say that $q$ is {\sl pure} if there is a subset $D$ of $Y$
such that (i) $D$ is dense in $Y$ and the restriction of $q$ to $D$ is injective, and (ii) for every net $(d_{\alpha})$ in $D$, if $d_{\alpha}\to\infty$ (i.e.
eventually escapes from every compact set in $Y$) then $e_{\alpha}\to \infty$ for any other net $(e_{\alpha})$ for
which $q(d_{\alpha})=q(e_{\alpha})$ for all $\alpha$.

\medskip
\noindent Condition (ii) is equivalent to requiring that if $d_{\alpha}\to\infty$ then the net
$q^{-1}(q(d_{\alpha}))$ converges to infinity in the Fell topology. Lemma 3.1 shows that condition (ii) is automatically satisfied whenever $q$ is a closed quotient map.

On the other hand, condition (i) implies that for every $x\in X$, the
interior of $q^{-1}(x)$ is either empty or consists of an isolated point. In fact condition (i) is equivalent to this when $Y$ is second countable.
To see this, let $D_0$ be the (countable) set of isolated points in $Y$ and  fix a countable base $\{U_i\}_{i\ge 1}$ for
$Y\setminus \overline{D_0}$. Inductively chose a point from each $U_i$ which does not belong to any fibre with non-empty interior
or any fibre previously chosen. This is possible because $U_i$ contains no isolated point and hence is an uncountable Baire space, and every
fibre intersects $U_i$ in a closed set with empty interior. Let $D_1$ be the set thus obtained, and set
$D=D_0\cup D_1$. Then $D$ is dense in $Y$ and $q$ restricted to $D$ is injective.

\medskip
\noindent Here is an alternative description of pure quotient maps  that will be useful later.
\medskip

\begin{lemma}\label{Lemma 3.2} Let $Y$ be a locally compact, $\sigma$-compact Hausdorff space and $q:Y\to X$ a quotient
map with $X$ Hausdorff. Then the following are equivalent:

(a) $q$ is pure;

(b) there is a subset $D$ in $Y$ such that (i) $D$ is dense and the restriction of $q$ to $D$ is injective, and (ii) whenever $Y=\bigcup_{i\ge 1} Y_i$ is a compact
decomposition of $Y$, the sets $D_i:=\{ d\in D: q(d)\in q(Y_i)\}$ have compact closure in $Y$.
\end{lemma}
\medskip
\begin{proof} (a)$\Rightarrow$(b). Let $W$ be the closure of $D_i$. If $W$ is not compact,
then there is a sequence $(d_j)$ in $D_i$ with $d_j\to\infty$ as $j\to\infty$. But for each $j$, there
exists $y_j\in Y_i$ such that $q(y_j)=q(d_j)$. Since $Y_i$ is compact, $y_j\not\to\infty$, contradicting condition
(ii) for pure quotient maps.

(b)$\Rightarrow$(a). Let $(d_{\alpha})$ be a net in $D$ with $d_{\alpha}\to\infty$. Let $(y_{\alpha})$
be any net in $Y$ such that $q(y_{\alpha})=q(d_{\alpha})$ for each $\alpha$. Then for any $Y_i$,
if $(y_{\alpha})$ were frequently in $Y_i$ then $(d_{\alpha})$ would frequently be in the compact closure
of $D_i$, contradicting the convergence of $(d_{\alpha})$ to infinity. Hence $(y_{\alpha})$ is eventually
outside each compact set $Y_i$, so $y_{\alpha}\to \infty$.
\end{proof}

\smallskip
\noindent Thus if $q$ is a pure quotient map and $Y=\bigcup_{i\ge 1} Y_i$ is a compact
decomposition of $Y$, then for each $Y_i$ we may find $Y_{d(i)}$ with $d(i)\ge i$ such that $D_i\subseteq Y_{d(i)}$.
\bigskip

\noindent The usefulness of pure quotient maps appears in the following converse to Proposition 2.2 (note that
because a fibre of a pure quotient map has at most a singleton as its interior, the fibre is compact if and only if its
boundary is compact).
\bigskip

\begin{thm}\label{Theorem  3.3.} Let $Y$ be a locally compact, $\sigma$-compact Hausdorff space and $q: Y\to X$ a pure quotient map with $X$ Hausdorff.
If $x\in X$ has a compact neighbourhood then $q$ is locally closed at $x$ and $q^{-1}(x)$ is compact.
\end{thm}

\medskip
\begin{proof} If $Y$ is compact then $q$ is closed and $q^{-1}(x)$ is compact for all $x\in X$, so we may suppose
that $Y$ is non-compact. Let $Y=\bigcup_{i\ge 1} Y_i$ be a compact decomposition for $Y$.
Suppose first that $q$ is not locally closed at $x$.  Let $W$ be a closed subset of $Y$ such that $x\in\overline {q(W)}$
but $W$ does not meet $q^{-1}(x)$. Then for any neighbourhood $U$ of $x$, $q(W)\cap U$ is not closed in $U$,
so $W\cap q^{-1}(U)$ is not contained in any of the compact sets $Y_i$. Thus for each $i$, there exists $w_i\in (W\cap q^{-1}(U))\setminus Y_i$.
By the density of $D$, there exists $d_i\in (D\cap q^{-1}(U))\setminus Y_i$. Then the set $F=\bigcup_i q^{-1}(d_i)$ is a saturated
subset of $Y$, and $F$ is closed by condition (ii) for pure quotient maps. Hence $q(F)$ is an infinite closed discrete subset of $U$,
so $U$ is not countably compact.

Now suppose instead that $q^{-1}(x)$ is non-compact, and let $U$ be a neighbourhood of $x$. Then for each $i\ge 1$, there
exists $y_i\in q^{-1}(x)\setminus Y_i$, so by the density of $D$, there exists $d_i\in q^{-1}(U)\setminus Y_i$.
Then the set $F=\bigcup_i q^{-1}(d_i)$ is a saturated
subset of $Y$, and $F$ is closed by condition (ii) for pure quotient maps. Hence $q(F)$ is an infinite closed discrete subset of $U$,
so again $U$ is not countably compact.
\end{proof}

\bigskip
\noindent
Pure quotient maps are not uncommon. Clearly, every locally compact $\sigma$-compact Hausdorff space $X$ admits the trivial pure quotient,
taking $Y=X$ and $q$ as the identity map, and we now show that every $k_{\omega}$-space $X$ admits a pure quotient.
Specifically, we show that if $Y$ is a locally compact, $\sigma$-compact Hausdorff space, and $q:Y\to X$ a quotient map with $X$ Hausdorff, then
$q$ can be `purified': that is, a locally compact $\sigma$-compact space $Y_*$ can be derived from $Y$ such that the quotient map $q_*:Y_*\to X$ is pure
(and induces the same topology on $X$ as $q$).

\bigskip
\noindent For a topological space $X$, let $\Cl(X)$ be the hyperspace of closed subsets of $X$ with the Fell topology.
A base for the Fell topology consists of the family of all sets
$$U(K, \Phi)=\{ S\in \Cl(X): S\cap K=\emptyset,\ S\cap O\ne\emptyset,\ O\in \Phi\}$$
where $K$ is a compact subset of $X$ and $\Phi$ is a finite family of open subsets of $X$.
Then $\Cl(X)$ is a compact space \cite[Lemma 1]{Fel}, and is Hausdorff if $X$ is locally compact \cite[Theorem 1]{Fel}.
The next lemma is presumably standard.

\bigskip
\begin{lemma}\label{Lemma 3.4.} Let $X$ be a locally compact $\sigma$-compact space. Then $\Cl(X)\setminus \{\emptyset\}$ is
$\sigma$-compact.
\end{lemma}

\medskip
\begin{proof} We may write $X=\bigcup_i X_i$ where each $X_i$ is compact. For each $i$, set $K_i=\{K\in \Cl(X):
K\cap X_i\ne\emptyset\}$. Let $(K_{\alpha})$ be a net in $K_i$. Then by the compactness of $\Cl(X)$, $(K_{\alpha})$ has a convergent
net $(K_{\beta})$ with limit $W$. But $W\cap X_i$ must be non-empty, since each $K_{\beta}$ meets $X_i$ and hence
$W\in K_i$. Thus $K_i$ is compact, and $\Cl(X)\setminus \{\emptyset\}=\bigcup_i K_i$ is $\sigma$-compact.
\end{proof}

\bigskip

\begin{lemma}\label {Lemma 3.5.} Let $Y$ be a locally compact $\sigma$-compact Hausdorff space and $q:Y\to X$ a quotient
map with $X$ Hausdorff. Let $\Cl(Y)$ be the hyperspace of closed subsets of $Y$ with the Fell topology, and let $Y_*$ be
the closure of the set $D=\{q^{-1}(x): x\in X\}$ in $\Cl(Y)\setminus \{\emptyset\}$. Then for each $F\in Y_*$ there exists $F'\in D$ such
that $F\subseteq F'$.
\end{lemma}

\medskip
\begin{proof} Let $(F_{\alpha})$ be a net in $D$ with limit $F\ne\emptyset$ in $\Cl(Y)$. Let $y\in F$ and set $F'=q^{-1}(q(y))$.
Then there is a net $(y_{\alpha})$ with $y_{\alpha}\in F_{\alpha}$ such that $y_{\alpha}\to y$. Hence $q(y_{\alpha})\to q(y)$. Let $y'\in F$. Then
there is another net $(y'_{\alpha})$ with $y'_{\alpha}\in F_{\alpha}$ such that $y'_{\alpha}\to y'$. Hence
$q(y_{\alpha})=q(y'_{\alpha})\to q(y')$, so $y'\in F'$. Thus $F\subseteq F'$.
\end{proof}

\bigskip
\noindent It follows from Lemma 3.5 that, in the context of the lemma, we may consistently define a map
$q_*: Y_*\to X$ by $q_*(F)=q(y)$ $(y\in F)$.

\bigskip
\begin{thm} \label {Theorem 3.6.} Let $Y$ be a locally compact $\sigma$-compact Hausdorff space and $q:Y\to X$ a quotient
map with $X$ Hausdorff. Let $\Cl(Y)$ be the hyperspace of closed subsets of $Y$ with the Fell topology, and let $Y_*$ be
the closure of the set $D=\{q^{-1}(x): x\in X\}$ in $\Cl(Y)\setminus \{\emptyset\}$. Then $Y_*$ is a locally compact,
$\sigma$-compact Hausdorff and the map $q_*: Y_*\to X$ is a pure quotient map inducing the same topology on $X$ as $q$.
\end{thm}

\medskip
\begin{proof} It follows from Lemma 3.4 that $Y_*$ is $\sigma$-compact.

Let $C$ be a subset of $X$. We must show that $q_*^{-1}(C)$ is closed if and only if $q^{-1}(C)$ is closed. Suppose first
that $q^{-1}(C)$ is closed, and let $F\in Y_*$ with $q_*(F)=y\notin C$. Then the complement of $q^{-1}(C)$ is an open
neighbourhood containing $F=q^{-1}(x)$, and thus no net in $q_*^{-1}(C)$ can converge to $F$.
Thus $q_*^{-1}(C)$ is closed.
Conversely, suppose that $q_*^{-1}(C)$ is closed. Let $(y_{\alpha})$ be a net in $q^{-1}(C)$ with limit $y$. Set $F_{\alpha}=q_*^{-1}(q(y_{\alpha}))$.
Then $F_{\alpha}\not\to\infty$, so by local compactness of $Y_*$ there exists a convergent subnet of $(F_{\alpha})$ with limit $F\in Y_*$, and $y\in F$. Then $F\in
q_*^{-1}(C)$ by assumption, and $q_*(F)=q(y)\in C$. Thus $y\in q^{-1}(C)$ as required.

Finally, it is immediate that the restriction of $q_*$ to $D$ is injective, while condition (ii) of pure quotient maps follows from Lemma 3.5.
\end{proof}

\bigskip
\noindent Note that  $q_*^{-1}(x)$ is the set consisting of  $q^{-1}(x)$ and  a family of closed subsets of the boundary of $q^{-1}(x)$, obtained as limits of nets of fibres;
see Lemma 3.5. (Thus in the example after Proposition 2.2, $D$ is homeomorphic to ${\R}$ and is closed in $\Cl(Y)\setminus \{\emptyset\}$).
If the boundary of $q^{-1}(x)$ is compact then $q_*^{-1}(x)$ is also compact because $Y_*$ is the closure of $D$ in $\Cl(Y)\setminus \{\emptyset\}$.
On the other hand, if the boundary of $q^{-1}(x)$ is non-compact then a net
of its closed subsets might converge to $\emptyset$ in the Fell topology, so $q_*^{-1}(x)$ might be compact or non-compact.

Apart from the fact that its fibres are more likely to be compact (or to have compact boundary) than those of $q$, a further
advantage that $q_*$ has over $q$ is that it is more likely to be locally closed.

\bigskip

\begin{prop}\label {Proposition 3.7.} Let $Y$ be a locally compact, $\sigma$-compact Hausdorff space and $q:Y\to X$ a quotient map
with $X$ Hausdorff. Let $q_*$ be the pure quotient constructed in Theorem 3.6, and let $x\in X$. If $q$
is locally closed at $x$, then so is $q_*$.
\end{prop}

\medskip
\begin{proof} Suppose that $q$ is locally closed at $x$ and let $C$ be a closed subset of $Y_*$ disjoint from
$q_*^{-1}(x)$. Set $W=\bigcup_{K\in C}K$. Then $W$ is disjoint from $q^{-1}(x)$. Let $(y_{\alpha})$ be a net
in $W$ with limit $y\in Y$. Then there is a net $(K_{\alpha})$ in $C$ with $y_{\alpha}\in K_{\alpha}$ for each
$\alpha$, and $K_{\alpha}\not\to\infty$, so by the local compactness of $Y_*$ there is a net
$(K_{\beta})$ converging to some $K$, which belongs to $C$ since $C$ is closed. Since $y\in K$, it
follows that $y\in W$, and hence that $W$ is closed. Since $q(W)=q_*(C)$, it follows that $x\notin \overline{q_*(C)} $.
\end{proof}

\bigskip\noindent In particular, if $q$ is closed then $q_*$ is closed. Note, however, that in this case
if the boundary of $q^{-1}(x)$ is non-compact then $x$ does not have a compact neighbourhood
in $X$ by Theorem 2.1, and hence $q_*^{-1}(x)$ is non-compact by Proposition 2.2 applied
to $q_*$.
\bigskip

\noindent Proposition 3.7 raises the question of whether $q_{**}$ might be locally closed at even more
points of $X$ than $q_*$ is, but we will see in the next section that $q_*$ already reaches the limit.

\bigskip

\section {An extension of Morita's theorem}

\bigskip

\noindent In this section, we use pure quotient maps to extend Morita's theorem to general $k_{\omega}$-spaces.

\medskip

Given a quotient map $q$ from a locally compact $\sigma$-compact Hausdorff space $Y$ onto
a Hausdorff space $X$, we partition $X$ into three sets as follows:

(i) $L_q=\{x\in X: \hbox{ $q$ is locally closed at $x$ and $q^{-1}(x)$ has compact boundary}\}$;

(ii) $F_q=\{ x\in X: \hbox{ $q$ is locally closed at $x$ and $q^{-1}(x)$ has non-compact boundary}\}$;

(iii) $N_q=\{ x\in X: \hbox{ $q$ is not locally closed at $x$}\}$.

\bigskip
\noindent We saw in Proposition 2.2 that if $x\in L_q$ then $x$ has a compact neighbourhood in $X$, and in Theorem 3.3 that if $q$ is a pure
quotient map then $L_q$ is precisely the set of points in $X$ which have a compact neighbourhood. We now show that if
$q$ is a pure quotient map, the sets $F_q$ and $N_q$ can similarly be characterised in terms of the topology of $X$.

\bigskip

\begin{thm} \label {Theorem 4.1.} Let $Y$ be a locally compact $\sigma$-compact Hausdorff space and $q: Y\to X$ a pure quotient map with $X$ Hausdorff.
Then $q$ is locally closed at $x\in X$ if and only if $x$ is a $k$-point.
\end{thm}

\bigskip

\begin{proof} We have seen that if $q$ is locally closed at $x$ then it is
pseudo-open at $x$, and hence $x$ is a $k$-point by Theorem 2.3.

Conversely, suppose that $q$ is not locally closed at $x$, and let $N$ be a closed subset of $Y$
disjoint from $q^{-1}(x)$ with $x\in\overline{q(N)}$.
Let $W$ be a closed subset of $Y$ disjoint from $q^{-1}(x)$ with $N$ in its interior. Let $D$ be the dense subset
of $Y$ from the pure quotient property, and set $M=W\cap D$ and $E=q(M)$. Then
$\overline M\supseteq N$, so $x\in\overline {E}\setminus E$. Let $Y=\bigcup_i Y_i$ be a compact decomposition
for $Y$ and set $K_i=q(Y_i)$. Then for any $i$, $q^{-1}(K_i\cap E)=\{ d\in M: q(d)\in K_i\}$ so
$q^{-1}(K_i\cap E)\subseteq Y_{d(i)}$ (in the terminology introduced after Lemma 3.2). Thus
$$K_i\cap E\subseteq q(Y_{d(i)}\cap M)\subseteq q(Y_{d(i)}\cap \overline M)\subseteq q(Y_{d(i)}\cap W) $$ and this latter
set is compact and does not contain $x$. Since every compact subset of $X$ is contained in some $K_i$ by hemicompactness,
it follows that there does not exist compact $K$ with $x\in\overline{E\cap K}$. Hence $x$ is not a $k$-point.
\end{proof}

\bigskip

\noindent Thus if $q$ is a pure quotient map, $q$ is locally closed at $x$ if and only if $q$ is pseudo-open at $x$. The following corollary is immediate from
Theorem 4.1.

\bigskip

\begin{cor}\label {Corollary 4.2.} Let $Y$ be a locally compact $\sigma$-compact Hausdorff space and $q: Y\to X$ a pure quotient map with $X$ Hausdorff.
Let $x\in X$. Then

(i) $x\in L_q$ if and only if $x$ has a compact neighbourhood in $X$;

(ii) $x\in F_q$ if and only if $x$ is a $k$-point without a compact neighourhood in $X$;

(iii) $x\in N_q$ if and only if $x$ is not a $k$-point.
\end{cor}

\bigskip
\noindent In particular, for any quotient map with domain a locally compact, $\sigma$-compact Hausdorff space and
Hausdorff range, $N_{q_{**}}=N_{q_*}$, answering the question raised in Section 3. It follows from Corollary 4.2 that for a $k_{\omega}$-space $X$,
the sets $L_{q_*}$, $F_{q_*}$, and $N_{q_*}$ are independent
of the original quotient map $q$, and in future we may use this notation even when there is no particular original quotient map in mind.

\bigskip
\noindent The next corollary is also immediate from Theorem 4.1 (applied to $q_*$) and Theorem 2.3 (applied to $q$). Recall that $H_q$ denotes the
set of points of $X$ at which $q$ is pseudo-open.

\bigskip
\begin{cor} \label {Corollary 4.3.} Let $Y$ be a locally compact $\sigma$-compact Hausdorff space and $q: Y\to X$ a quotient map with $X$ Hausdorff.
For $x\in X$, $x\in H_q$ if and only if $q_*$ is locally closed at $x$. Hence $q$ is pseudo-open if and only if $q_*$ is closed.
\end{cor}

\bigskip
\noindent We turn now to consider the topology of the sets $L_{q}$, $F_{q}$, and $N_{q}$ when $q$ is pure (in which case
$H_q=L_q\cup F_q$ by Corollary 4.3). It follows from Corollary 4.2 that $L_q$ is open, and hence that $F_q\cup N_q$ is closed in $X$.

Recall that a point $x$ in a Tychonoff space $X$ is a P-point if the intersection of any countable family of neighbourhoods
of $x$ is a neighbourhood of $x$. If every point is a P-point then $X$ is a P-space. If $U$ is an open subset of $X$ and $x$ is a P-point in
$U$ then $x$ is P-point in $X$. Conversely, a P-point in $X$ is a P-point in any subspace
of $X$, and hence the set of P-points of $X$ forms a P-space. If $X$ is a compact P-space then $X$ is finite \cite[4K]{GJ}.
The next theorem shows that $F_q$ is contained
in the set of P-points of $F_q\cup N_q$. For this we need the following observation.

Let $Y$ be a locally compact $\sigma$-compact Hausdorff space and $q: Y\to X$ a quotient map with $X$ Hausdorff.
Then $q$ extends to a continuous map $\overline q: \beta Y\to \beta X$, and  $\overline q$ is continuous and surjective from
a compact Hausdorff space, and is therefore also a quotient map. If $x\in X$ and $q$ is not locally closed at $x$, then there is a closed set
$N\subseteq Y$ disjoint from $q^{-1}(x)$ such that $x\in \overline{q(N)}$. Hence there is a point $y\in \overline{N}^{\beta Y}$
such that $\overline q(y)=x$.

\bigskip

\begin{lemma}\label {Lemma 4.4.} Let $Y$ be a locally compact $\sigma$-compact Hausdorff space and $q: Y\to X$ a pure quotient map with $X$ Hausdorff.
Let $x\in N_q$ and let $y\in \beta Y\setminus Y$ such that $\overline q(y)=x$. Let $V$ be any neighbourhood of $x$ in $X$.
Then $Y$ has a zero set $Z\subseteq q^{-1}(V)$ with $y\in \overline{Z}^{\beta Y}$.
\end{lemma}

\bigskip
\begin{proof} Let $U\subseteq V$ be cozero set  containing $x$ with complement $W$ in $X$. Then $q^{-1}(W)$ is a saturated zero subset in $Y$.
If $Z$ is a zero set in $Y$ contained in $q^{-1}(W)$ then $q(Z)$ is contained in $W$, so $y\notin\overline{Z}^{\beta Y}$. Hence, by the $z$-ultrafilter
property \cite[2.6 Theorem]{GJ}, there is a zero set $Z$ in $Y$ with $y\in\overline{Z}^{\beta Y}$ such that $Z\cap q^{-1}(W)$ is empty.
\end{proof}

\bigskip
\begin{thm} \label {Theorem 4.5.} Let $Y$ be a locally compact $\sigma$-compact Hausdorff space and $q: Y\to X$ a pure quotient map with $X$ Hausdorff.

(i) If $x\in X$ is a P-point then $q$ is locally closed at $x$.

(ii) Let $W$ be the set of P-points in $ F_q\cup N_q$ and $V$ the interior of $F_q\cup N_q$ in $X$.
Then $$V\cap W\subseteq F_q\subseteq W.$$

(iii) $F_q$ is a P-space and has empty interior in $X$.
\end{thm}

\bigskip
\begin{proof} (i) (This does not require $q$ to be pure).
Let $C$ be a closed subset of $Y$ disjoint from $q^{-1}(x)$. Then $$C=\bigcup_{n\ge 1} C\cap K_n,$$ where $Y=\bigcup_{n\ge 1} K_n$
is a compact decomposition for $Y$, so $q(C)=\bigcup_{n\ge 1} q(C\cap K_n)$. Since $q(C\cap K_n)$ is compact, hence closed in $X$,
the P-point $x$ does not lie in $\overline{q(C)}$. Thus $q$ is locally closed at $x$.

(ii) It follows from (i) that $F_q\supseteq V\cap W$. Now suppose that $x_0\in F_q\cup N_q$ is not a P-point of $F_q\cup N_q$.
Let $\{E_i: i\ge 1\}$ be a countable family of closed sets in $F_q\cup N_q\setminus\{ x_0\}$ with $x_0\in\overline {\bigcup_i
E_i}$. Fix $i$, and let $U_i$ be a closed neighbourhood of $x_0$ in $X$ that does not meet $E_i$.
For each $x\in E_i$ choose $y\in q^{-1}(x)$ with $y\notin Y_{i+1}$  if this is possible  (as it is for each $x\in F_q$
since $q^{-1}(x)$ is non-compact). Otherwise, by the observation preceding Lemma 4.4,
there exists $y\in \beta Y\setminus Y$ such that $\overline q(y)=x$. By Lemma 4.4,
there is a zero set $Z$ in $Y$ with $y\in \overline Z^{\beta Y}$ and $Z\cap q^{-1}(U_i)$ empty,
and we may also arrange that $Z\cap Y_{i+1}$ is empty.

Let $N_i'$ be the union of the collection of elements $y$ and zero sets $Z$ thus obtained, and let $N_i$ be the closure
of  $N_i'$ in $Y$. Then $N_i$ does not meet $Y_i$ since $Y_i$ is contained in the interior of $Y_{i+1}$, and $\overline{q(N_i)}\supseteq E_i$,
but $q(N_i)\cap U_i$ is empty (so in particular $x_0\notin q(N_i)$). Set
$N=\bigcup N_i$. Then $N$ does not meet $q^{-1}(x_0)$, and $N$ is closed in $Y$ because, for any $Y_j$, only finitely many of the sets $N_i$
meet $Y_j$, so $N\cap Y_j$ is closed.
But $\overline {q(N)}\supseteq\bigcup_i E_i$, so $x_0\in\overline {q(N)}\setminus q(N)$. Hence $x_0\in N_q$, so $F_q\subseteq W$.

(iii) $F_q$ is a P-space by (ii). Let $U$ be the interior of $F_q$ in $F_q\cup N_q$. Then $O:=q^{-1}(U)$ is open in $Y$, hence a locally
compact Hausdorff space, and $q|_O$ is a quotient map. Thus $U$ is a $k$-space, so a subset of $U$ is closed if and only if
its intersection with every compact subset of $U$ is closed. But $U$ is P-space, so compact subsets are finite \cite[4K]{GJ}.
Thus every subset of $U$ is closed, so $U$ is discrete. Hence each point in $U$ lies in the closure of $L_q$ (for
otherwise it would be an isolated point of $X$ and would belong to $L_q$ itself). Thus $F_q$ has empty interior in $X$.
\end{proof}

\medskip
\noindent Theorem 4.5 is an extension of Morita's theorem (1.1) because if $q$ is a closed quotient map then
$q_*$ is closed (Proposition 3.7), so $N_q$ and $N_{q_*}$ are both empty. Furthermore $F_q=F_{q_*}$ by the remark after Proposition 3.7,
so $F_{q}$ is closed, and thus $\sigma$-compact. Being a P-space by Theorem 4.5, $F_{q}$ must be a countable discrete set in the closure
of $L_q$, and thus we recover Morita's theorem.

In some cases we can go further. Recall that a topological space $X$ has {\sl countable tightness} at $x\in X$
if for each $A\subseteq X$ with $x\in \overline A$, there is a countable subset $B\subseteq A$ with $x\in\overline B$;
and $X$ is {\sl countably tight} if it is has countable tightness at every point. Clearly every metric space is countably tight.

\bigskip
\begin{cor} \label {Corollary 4.6.} Let $Y$ be a locally compact $\sigma$-compact Hausdorff space and $q: Y\to X$ a pure quotient map with $X$ Hausdorff.

(i) If $x\in F_q$ is a point of countable tightness in $X$ then $x$ is isolated in $F_q\cup N_q$ and lies in the closure of $L_q$.

(ii) If $X$ is countably tight then $N_q$ is closed in $X$, and $F_q$ is discrete and lies in the closure of $L_q$.
\end{cor}

\medskip

\begin{proof} (i) follows immediately from the fact that $x$ is a P-point in $F_q\cup N_q$, and (ii) follows from (i).
\end{proof}

\medskip

\noindent Corollary 4.6 seems to be new even for quotients of locally compact subsets of ${\R}^n$.

We have just seen
that if $X$ is countably tight then there are no points of $F_q$ in the interior of $N_q\cup F_q$, but if $X$ is not countably tight,
such points may occur. For example, let $X$ be the quotient of $\omega_1+1\times S_{\omega}$ (where $S_{\omega}$ is the Arkhangel'skii-Franklin space \cite{AF}
obtained by identifying $\{\omega_1\}\times S_{\omega}$ to a point $x_0$. Then $L_q$ is empty and $F_q=\{x_0\}$.

Not every P-point of $F_q\cup N_q$ need belong to $F_q$. For example, let $Y={\R}$ and for $n\ge 1$,
identify the pairs of points $n$ and $1/n$. Let $X$ be the resulting quotient space and $q:Y\to X$ the quotient map, which is obviously pure.
Then $F_q\cup N_q$ consists of the singleton $\{ q(0)\}$, but $F_q$ is empty.
The following example, however, has the satisfying feature that $F_q$ consists precisely of the P-points in $F_q\cup N_q$.

\bigskip
\noindent {\bf Example 4.7.} Let $Z$ be a locally compact $\sigma$-compact Hausdorff space and $V$ a closed subset
with empty interior (so that $V\subseteq \overline {Z\setminus V}$). We shall find
a locally compact $\sigma$-compact Hausdorff space $Y$ and a pure quotient map $q: Y\to X$ with $X$ Hausdorff
such that $V$ is homeomorphic to $F_q\cup N_q$, with $F_q$ corresponding  to the P-points in $V$.
\bigskip

\noindent Let $Y=\bigcup_i Z_i$ be the disjoint union of countably many copies
of $Z$. Set $W=\bigcup_i V_i$, and for each $i$ let $\phi_i$ be a homeomorphism from $V$ to $V_i$. Let $O$ be the complement of $W$
in $Y$. Let $\sim$ be the equivalence relation defined on points $y\in Y$ by $y\sim y'$ if and only if either (i) $y=y'$; or (ii) $y, y'\in W$ and there exists $v\in V$
such that $\phi_i(v)=y$ and $\phi_j(v)=y'$ for some $i$ and $j$. Set $X=Y/\sim$ and let $q$ be the quotient map. Then $q$ is pure
because $V$ has empty interior in $Z$.
Note that $q(W)$ is homeomorphic to $V$, and that $q|_{V_i}$ is a homeomorphism from $V_i$ onto $q(W)$ for each $i$. For ease of notation we
will identify $q(W)$ with $V$, so that $\phi_i: V\to V_i$ is the inverse of $q|_{V_i}$.

Clearly $q|_O$ is a homeomorphism from $O$ onto $q(O)$, so if $x\in q(O)$ then $x$ is a point of local compactness.
Now let $x\in V$ and suppose first that $x$ is not a P-point in $V$. Let $\{N_i\}$ be a countable family of closed
subsets of $V$ such that $x\notin N_i$ for all $i$ but $x\in\overline {\bigcup_i N_i}$. Set $N=\bigcup_i \phi_i(N_i)$. Then $N$
is a closed subset of $Y$ disjoint from $q^{-1}(x)$, but $x\in \overline {q(N)}$. Hence $x\in N_q$.

Now suppose that that $x$ is a P-point in $V$. Let $N$ be a closed subset of $Y$ not
meeting $q^{-1}(x)$. Then the sets $q(N\cap V_i)$ are a countable family of closed set in $V$ (since $q|_{V_i}$ is a homeomorphism), and they
each miss $x$, so $M:=\overline {\bigcup_i q(N\cap V_i)}$ also misses $x$. Set $N'=N\cup \bigcup_i \phi_i(M)$.
Then $N'$ is a closed saturated set in $Y$ containing $N$ and missing
$q^{-1}(x)$, so $x\notin \overline{q(N)}$. Hence $x\in F_q$.

\bigskip

\noindent Thus in Example 4.7, $N_q$ is closed in $X$ if and only if every P-point in $V$ is an isolated point of $V$.
If $V$ has a non-isolated P-point, for example if $V=\omega_1+1$ (and $Z$ is the cone over $V$), then $N_q$ is not closed. If $Z=
\beta {\NN}$ and $V=\beta{\NN}\setminus {\NN}$, then it depends on the set theoretic axioms adopted as to
whether $F_q$ is empty or dense in $F_q\cup N_q$.

\bigskip

\section {Arkhangelskii's question and Fr\'echet-Urysohn spaces}

\bigskip

\noindent Finally, we turn to consider first countable and Fr\'echet-Urysohn spaces and Arkhangel'skii's question
(for a vast amount of information on this whole area, see \cite{Siwiec}).

As mentioned in the Introduction, a point $x$ in a topological space $X$ is a {\sl Fr\'echet-Urysohn} point if whenever $E\subseteq X$
with $x\in\overline E$ there is a sequence $(v_n)$ in $E$ with $v_n\to x$. Clearly every point of first countability
is a Fr\'echet-Urysohn point. If every point is a Fr\'echet-Urysohn point
then $X$ is a {\sl Fr\'echet-Urysohn} space. A Fr\'echet-Urysohn point in a $k_{\omega}$-space is easily seen to be a
a $k$-point (Proposition 5.1(ii)) and thus belongs either to $L_{q_*}$ or to $F_{q_*}$.
Ordman showed that a point of first countability must belong to $L_{q_*}$ (see \cite[p. 113]{Fr}) but a general
Fr\'echet-Urysohn point may belong to $F_{q_*}$. Being a point of countable tightness, however, it must
belong to the closure of $L_{q_*}$, and the set of such points must be countable (Corollary 4.6(i)).

The importance of Fr\'echet-Urysohn spaces is that every quotient of a compact first countable (or Fr\'echet-Urysohn)
space is Fr\'echet-Urysohn (see Corollary 5.2 below).
For example, let $Y=[0,1]\times [0,1]$ with the order topology of the lexicographic order and let $V=[0,1]\times \{0,1\}$.
Let $q:Y\to X$ be the quotient map that  takes $V$ to a point $x_0$. Then $Y$ is compact, Hausdorff, and first countable, and $V$ is closed but not a zero-set, so
$X$ is compact, Hausdorff and Fr\'echet-Urysohn but not first countable at $x_0$.

Recall that a topological space $X$ is {\sl sequential} if every sequentially closed subset of $X$ is closed. Every Fr\'echet-Urysohn space is
sequential, but there are compact Hausdorff sequential spaces which are not Fr\'echet-Urysohn.
More generally, every sequential space has countable tightness; and the famous Moore-Mrowka problem -- shown by Ostaszewski,
Fedorcuk, Balogh, Dow, Eisworth, {\sl et al.} to depend on set-theoretic axioms --
was whether there exists a compact Hausdorff space of countable tightness which is not sequential.

If $Y$ is a second countable locally compact Hausdorff then $\Cl(Y)$ is also second countable \cite[Remark 3, p. 474]{Fel}, but the same does not hold for
first countability. Indeed, if $Y$ is a locally compact Hausdorff space, then the following three conditions are equivalent: (a) $\Cl(Y)\setminus\{\emptyset\}$
is first countable; (b) $Y$ is both hereditarily
Lindelof and hereditary separable (for a locally compact space, first countability is a consequence of being hereditarily Lindelof) (c) $\Cl(Y)\setminus \{\emptyset\}$
is countably tight \cite{Hou}, \cite[Corollary 2.16]{CHV}.
The unit square with the lexicographic order (above) is neither separable nor hereditarily Lindelof.
The passage from $Y$ to $Y_*$ in Theorem 3.6 thus preserves second countability but it is not clear that it will preserve first countability or
the Fr\'echet-Urysohn property (although we
do not have specific counter-examples). This problem, however, is easily circumvented.

\bigskip

\begin{prop}  \label {Proposition 5.1.} Let $Y$ be a locally compact, $\sigma$-compact Hausdorff space and $q:Y\to X$ a quotient map
with $X$ Hausdorff.

(i) If $Y$ is a Fr\'echet-Urysohn space and $q_*$ is locally closed at $x$ then $x$ is a Fr\'echet-Urysohn point.

(ii) Suppose that $q$ is pure. If $x\in X$ is a Fr\'echet-Urysohn point then $q$ is locally closed at $x$.
\end{prop}

\medskip

\begin{proof} (i) Let $V$ be a subset of $X$ and suppose that $x\in\overline V$. With $E:=q_*^{-1}(V)$, we have $\overline E \cap q_*^{-1}(x)\ne\emptyset$
because $q_*$ is locally closed at $x$. Let $(C_{\alpha})$ be a net of points in $E$ converging to some
$C\in q_*^{-1}(x)$. Let $y\in C$. Then there exists a net $(y_{\alpha})$, with $y_{\alpha}\in C_{\alpha}$ for each $\alpha$, such that
$\lim_{\alpha} y_{\alpha}=y$. Then the set $\{y_{\alpha}\}\subseteq q^{-1}(V)$, and hence $ \overline{q^{-1}(V)}\cap
q^{-1}(x)\supseteq \{y\}$. Since $Y$ is a Fr\'echet-Urysohn space, there is a sequence $(y_n)$ in $\overline{q^{-1}(V)}$
converging to $y$. Thus $(q(y_n))$ gives the required sequence in $V$ converging to $x$.

(ii) Let $x\in X$ be a Fr\'echet-Urysohn point and let $E\subseteq X$ with $x\in\overline E$. Then by
assumption there is a sequence $(x_n)$ in $E$ such that $x_n\to x$.
Let $K$ be the compact set $\{x\}\cup\{x_i: i\ge 1\}$. Then $x\in \overline {E\cap K}$, so $x$ is a $k$-point, and hence
$q$ is locally closed at $x$ by Theorem 4.1.
\end{proof}

\bigskip
\begin{cor} \label{Corollary 5.2.}  Let $Y$ be a Fr\'echet-Urysohn locally compact, $\sigma$-compact, Hausdorff space and
$q:Y\to X$ a quotient map with $X$ Hausdorff.

(i) The open set $H_{q_*}=L_{q_*}\cup F_{q_*}$ is the set of Fr\'echet-Urysohn points of $X$.

(ii) $X$ is a Fr\'echet-Urysohn space if and only if $q_*$ is closed.

(iii) If $X$ is locally compact then $X$ is a Fr\'echet-Urysohn space.
\end{cor}
\medskip

\begin{proof} (i) follows from Proposition 5.1 and Corollary 4.6;
(ii) from Proposition 5.1; and (iii) from Proposition 5.1 and Theorem 3.3.
\end{proof}

\bigskip

\begin{cor} \label {Corollary 5.3.} Let $Y$ be a locally compact, $\sigma$-compact, Hausdorff space and
$q:Y\to X$ a quotient map with $X$ Hausdorff and a Fr\'echet-Urysohn space. Then

(i) $q_*$ is closed and $L_{q_*}$ is a dense open subset of $X$ with discrete complement;

(ii) if $Y$ is hereditarily Lindelof then $X$ is first countable  at each point of $L_{q_*}$.
\end{cor}

\medskip

\begin{proof} (i) The first statement follows from Theorems 2.1 and 3.6 and Proposition 5.1(ii).

(ii) For second
statement, note (as already mentioned) that a locally compact hereditarily Lindelof space is first countable. Furthermore, the property
of being hereditarily Lindelof is equivalent, in the presence of local compactness, to every open subset being
$\sigma$-compact. The set $L_{q_*}$ is locally compact and open, and
each open subset of $L_{q_*}$ inherits the property of $\sigma$-compactness from its inverse image under $q$. Hence
$L_{q_*}$ is first countable.
\end{proof}

\bigskip

\noindent {\bf Arkhangel'skii's question.} In \cite[Problems 5.14 and 5.15]{Arh}, Arkhangel'skii asked for a Hausdorff quotient of
a locally compact second countable Hausdorff space which is (countable and) Fr\'echet-Urysohn but nowhere first countable.
A locally compact second countable Hausdorff space is perfectly normal (i.e. every open
set is an $F_{\sigma}$), and the same is true in every Hausdorff quotient, so every point in a quotient is a $G_{\delta}$, and it is well known that a $G_{\delta}$ point
with a compact neighbourhood is first countable. Thus the required example had to have $L_{q_*}$ empty
and $F_{q_*}\cup N_{q_*}$ as a Fr\'echet-Urysohn space, along the lines of the space $S_{\omega}$ mentioned
in the Introduction.

In the absence of a structure theory for $k_{\omega}$-spaces, Arkhangel'skii's question is very pertinent, but in the course of this
paper we have seen multiple reasons why such an example cannot exist: for instance, every Fr\'echet-Urysohn
point in a $k_{\omega}$-space lies in the closure of $L_{q_*}$ (Corollary 4.6(ii)).
As we have mentioned, the question had, in fact, already been answered negatively by Siwiec in 1976 \cite{Siw} who
showed (among other things) that if a $k_{\omega}$-space is Fr\'echet-Urysohn and has a $k_{\omega}$-decomposition of compact
metric spaces then it is the closed image of a locally compact separable metric space (and hence Morita's theorem implies
that $L_{q_*}$ is dense).

Since a locally compact second countable space is hereditarily Lindelof, Corollary 5.3(i) and (ii) gives a more general negative answer
to Arkhangel'skii's question. If the Continuum Hypothesis is assumed, a still more general answer can be given.

\bigskip
\begin{thm} \label {Theorem 5.5.} (CH) Let $X$ be a sequential $k_{\omega}$-space.
Then $X$ is first countable at a dense subset of points if and only if the set of non-$k$-points of $X$ has empty interior.
\end{thm}

\bigskip

\begin{proof} The set $L_{q_*}$ of points with compact neighbourhoods is open and hence sequential.
Since $X$ is sequential it is countably tight, so it follows from Corollary 4.6(ii) that $N_{q_*}$ has empty interior
if and only $L_{q_*}$ is dense. If $X$ is first countable at $x\in X$ then $x\in L_{q_*}$, see \cite[p. 113]{Fr}. Thus
if $X$ is first countable at a dense subset of points, $N_{q_*}$ has empty interior. Conversely,
if $N_{q_*}$ has empty interior, then $L_{q_*}$ is dense in $X$, and
Arkhangelski showed in 1970 (\cite{Arhan}; see also \cite[p. 381]{Ar})  that under (CH) every (locally) compact Hausdorff sequential space
is first countable at a dense subset of points.
\end{proof}

\bigskip

\noindent In particular, under (CH), every Fr\'echet-Urysohn $k_{\omega}$-space is first countable at a dense subset of points. On the other
hand, in 1987 Malyhin used the method of forcing to exhibit a compact Hausdorff Fr\'echet-Urysohn which is nowhere first countable \cite{Maly}.
Thus Corollary 5.3 may be about as far as one can go in ZFC in answer to Arkhangel'skii's question.

\bigskip
\noindent {\bf Acknowledgements:} We are grateful to Eric Wofsey and Henno Brandsma for help on Math Stack Exchange
at early stages of this enquiry.

\end{document}